\documentclass{article}
\usepackage{graphicx} 
\usepackage[utf8]{inputenc}
\usepackage[affil-it]{authblk}
\usepackage{graphicx} 
\usepackage{subcaption}
\usepackage{comment,amsmath,amsbsy,amssymb,graphicx,dsfont,upgreek,textcomp,braket,setspace,blindtext,hyperref,verbatim,amsthm,mathrsfs,mathtools,graphicx,float,enumerate,cleveref,cases,bigints,booktabs}
\usepackage[font={scriptsize,it}]{caption}
\usepackage[margin=1in]{geometry}
\hypersetup{
    colorlinks=true, 
    linktoc=all,     
    linkcolor=blue,  
}
\usepackage{tocloft} 
\usepackage[table,svgnames]{xcolor}
\usepackage{tikz}
\usetikzlibrary{arrows}
\usepackage[toc,page]{appendix}
\usepackage{xcolor}
\DeclareMathOperator{\sech}{sech}

\Crefname{appsec}{appendix}{appendices}
\numberwithin{equation}{section}

\newtheorem{proposition}{Proposition}[section]

\title{\Large Nonexistence of traveling wave solutions in the fractional Rosenau-Hyman equation via homotopy perturbation method\vspace{-1ex}}
\author{Brian Choi \vspace{-2ex}}
 
\author{Brian Choi\thanks{Corresponding author. United States Military Academy, \texttt{choigh@bu.edu}}
\vspace{-2ex}  }

\date{}
\begin{document}
\maketitle\vspace{-4ex}
\maketitle
\begin{abstract}
We apply the homotopy perturbation method to construct series solutions for the fractional Rosenau-Hyman (fRH) equation and study their dynamics. Unlike the classical RH equation where compactons arise from truncated periodic solutions, we show that spatial nonlocality prevents the existence of compactons, and therefore periodic traveling waves are considered. By asymptotic analyses involving the Mittag-Leffler function, it is shown that the quadratic fRH equation exhibits bifurcation with respect to the order of the temporal fractional derivative, leading to the eventual pinning of wave propagation. Additionally, numerical results suggest potential finite time blow-up in the cubic fRH. While HPM proves effective in constructing analytic solutions, we identify cases of divergence, underscoring the need for further research into its convergence properties and broader applicability.
\end{abstract}



\section{Introduction}

Fractional differential equations (FDEs) have emerged as a powerful tool in modeling complex systems arising from nonlocal spatiotemporal interactions. While classical integer-order differential equations fall short by assuming instantaneous rates of change, FDEs account for memory effects and long-range interactions, making them particularly well-suited for describing anomalous diffusion \cite{metzler2000random}, viscoelasticity \cite{mainardi2022fractional}, and various phenomena in engineering, physics, and biology \cite{ionescu2017role} with applications in diverse fields, including fluid mechanics \cite{kulish2002application}, control theory \cite{monje2010fractional}, and signal processing \cite{magin2011fractional}, where long-range dependencies play a crucial role. Despite significant progress, the development of analytical and numerical methods for solving FDEs remains an ongoing challenge, particularly in nonlinear wave dynamics, keeping this field an active area of research in applied mathematics.

The homotopy perturbation method (HPM), an algorithmically simple method in obtaining exact series solutions, has gained significant attention for its efficiency in handling nonlinear problems. HPM, introduced in \cite{he1999homotopy,he2000coupling}, combines the homotopy technique with traditional perturbation methods to construct approximate solutions through a rapidly convergent series expansion; see \Cref{background} for details. Unlike conventional perturbation methods that require small parameters, HPM does not impose restrictive assumptions. Over the years, researchers have applied HPM to a wide range of FDEs: time-fractional delay FDEs \cite{farhood2023homotopy}, time-fractional nonlinear Schr\"odinger equation \cite{agheli2017analytical}, KdV-like equations \cite{anil2011he,yildirim2010analytical}, heat transfer \cite{ganji2006application}, and time-fractional predator-prey model \cite{owolabi2024laplace}, among others. See references therein for more applications of HPM.

The classical Rosenau-Hyman (RH) equation, introduced to study pattern formation in liquid drop dynamics \cite{rosenau1993compactons,rosenau2000compact,rosenau2005almost} by incorporating degenerate nonlinear dispersion, admits periodic wave trains, and under suitable truncations, these periodic structures can be reinterpreted as solitary waves with compact spatial support (compactons). One of the fundamental mechanisms that gives rise to the compact structure is the interplay between degenerate nonlinear dispersion and nonlinear convection, which is in stark contrast to the KdV-type equations, governed by linear dispersion relations, that give rise to solitons with exponential tails. The absence of these exponential tails in RH compactons ensures that energy propagation remains strictly localized, well-suited for physical systems designed to minimize energy dissipation or radiation.

In this paper, we demonstrate that HPM is a robust and efficient analytical method (see \Cref{method} for brief descriptions on other methods) that effectively handles the nonlocal and nonlinear structure of the fractional Rosenau-Hyman equation (fRH)
\begin{equation}\label{k22}
    \partial_t^\beta u + \partial_x (u^n) - \partial_x (-\partial_{xx})^{\frac{\alpha}{2}} (u^n) = 0,\ \alpha \in (0,2],\ \beta \in (0,1], 
\end{equation}
where the classical RH equation, or $K(n,n)$, is recovered at $\alpha = 2,\ \beta = 1$; see \Cref{background} for definitions. Our main contributions are as follows: 1. construction of analytic series solutions to the fractional $K(n,n),\ n=2,3$, convergent in time; 2. proof of nonexistence of traveling waves at uniform speeds under spatial nonlocality; 3. rigorous and numerical studies of both short-time and long-time behavior of the fRH equations with emphases on bifurcation with respect to nonlocal parameters that shows drastically distinct behavior relative to the classical RH equation. As such, we apply HPM to understand the qualitative behavior of fRH equations. The application of HPM is motivated by \cite{odibat2008compact2,odibat2007solitary,momani2007homotopy,odibat2008compact} where our construction of exact solutions directly extends \cite{odibat2008compact2,odibat2007solitary}. What is not studied in the aforementioned references, however, is the dynamical behavior of the solutions to fRH, in which we fill the gap in the literature. Most notably, for the fractional $K(2,2)$, it is not only shown that the phase evolution of periodic waves is sublinear in time for $t \ll 1$, consistent with the previous work, but is also shown that the traveling of nonlinear waves becomes pinned as $t \rightarrow \infty$ where the long-time dynamics undergoes bifurcation as $\beta$ crosses $\frac{1}{2}$ (see \Cref{dynamics_k22}). Similar phenomena of varying propagation speeds depending on the order of the Caputo derivative were studied on the fractional diffusion-wave equation \cite{luchko2013fractional,luchko2013propagation}. For the fractional $K(3,3)$, a numerical evidence of finite time blow-up is given, in stark contrast to the classical $K(3,3)$.

This paper is organized as follows. In \Cref{background2}, a revision of fractional calculus is given along with a proof of nonexistence of traveling waves under the Riesz derivative. In \Cref{method2}, HPM is revised. A numerical evidence of divergence of HPM in certain regimes is provided, calling for further research of the method's convergence and applicability. In \Cref{fk22,fk33}, series solutions are constructed with which their dynamical properties are studied.
\section{Mathematical background}\label{background2}
\subsection{Fractional calculus}\label{background}
A brief review of fractional calculus is presented for readers' convenience. See \cite{podlubny1998fractional,kilbas2006theory,mainardi2022fractional,baleanu2012fractional} for more thorough discussion. Key notions include, among others, fractional integrals, derivatives, and the Mittag-Leffler function. For simplicity, we assume that functions are sufficiently smooth with appropriate decay.

For $f:[0,\infty) \rightarrow \mathbb{R}$ and $\gamma,\beta > 0$, define the fractional integral and the Caputo derivative, respectively, as
\begin{equation*}
\begin{split}
J^\gamma f (t) &= \frac{1}{\Gamma(\gamma)} \int_0^t f(\tau) (t-\tau)^{\gamma-1}d\tau,\\
\partial_t^\beta f (t) &= J^{\lceil \beta \rceil- \beta}\frac{d^{\lceil \beta \rceil}}{dt^{\lceil \beta \rceil}} f (t) = \frac{1}{\Gamma(\lceil \beta \rceil- \beta)} \int_0^t \frac{d^{\lceil \beta \rceil}f}{dt^{\lceil \beta \rceil}} (\tau) (t-\tau)^{\lceil \beta \rceil- \beta-1} d\tau.
\end{split}
\end{equation*}
The scalar-valued FDE $\partial_t^\beta u = F(u,t)$ is formally equivalent to, and mathematically valid with certain hypotheses (see \cite{daftardar2007analysis} for precise statements),
\begin{equation}\label{fde}
    u(t) = u(0) + \frac{1}{\Gamma(\beta)}\int_0^t F(u(\tau),\tau)(t-\tau)^{\beta - 1} d\tau.
\end{equation}
The Caputo derivative is often used to measure memory effects that depend on the history of the time-dependent system. In the spatial domain $\mathbb{R}^d$ where there exists no preferred direction, the Riesz derivative is often used to model nonlocality, measured by $\alpha \in (0,2]$, by extending the Laplacian via the Fourier transform as
\begin{equation*}
    \mathcal{F}[(-\Delta)^{\frac{\alpha}{2}} f] (\xi) = |\xi|^{\alpha} \mathcal{F} [f] (\xi),\ \xi \in \mathbb{R}^d.
\end{equation*}
The Mittag-Leffler function defined as $E_{\alpha,\beta}(z) = \sum\limits_{n=0}^{\infty} \frac{z^n}{\Gamma(\alpha n + \beta)}$ is an entire function for $\alpha,\beta > 0$ where for shorthand, $E_{\alpha}(z) := E_{\alpha,1}(z)$. If $
\alpha \in (0,2),\ \beta \in \mathbb{C}$, then it obeys the asymptotic formula \cite[Theorem 1.4]{podlubny1998fractional}
\begin{equation}\label{ml_decay}
    E_{\alpha,\beta}(-x) =  \sum_{k=1}^{N}\frac{x^{-k}}{\Gamma(\beta - \alpha k)} +O_N(|x|^{-(N+1)}),\ x>0,\ x \rightarrow \infty,
\end{equation}
for any positive integer $N$. For $t>0$, it is known that $E_{\alpha,\beta}(-t^\alpha)$ is completely monotone for $\alpha,\beta > 0$ if and only if $0 < \alpha \leq 1,\ \beta \geq \alpha$. Due to complete monotonicity, $E_{\alpha}(-t^\alpha)$ has no zeros when $\alpha \in (0,1]$. If $\alpha \in (1,2)$, the function has finitely many zeros where the number of zeros increases to infinity as $\alpha \rightarrow 2$; note that $E_2(-t^2) = \cos (t^2)$. If $\alpha \in [2,\infty)$, the function has infinitely many zeros. See \cite{rae/1337001380,merkle2014completely} for more details on complete monotonicity and \cite{duan2014zeros} on the zeros of the Mittag-Leffler function.
\subsection{Nonexistence of traveling compactons}
Generally, compactons are not preserved under the flows generated by nonlocal evolution operators, and therefore are not considered in this paper; instead, we consider spatially periodic solutions. By \Cref{no_travel}, we consider the trial functions \eqref{ansatz1}, \eqref{ansatz2} in $K(2,2)$ and $K(3,3)$, respectively, as periodic functions, not as compactly-supported functions, when applying HPM.
\begin{proposition}\label{no_travel}
    For $m,n \in \mathbb{N}$ and $\alpha \in (0,2)$, consider
    \begin{equation}\label{paley-wiener}
    \partial_t u + \partial_x (u^m) - \partial_x (-\Delta)^{\frac{\alpha}{2}} (u^n) = 0.
    \end{equation}
If there exists $f \in L^1(\mathbb{R}) \cap L^\infty(\mathbb{R})$ that is compactly supported, and $u(x,t) = f(x-ct)$ where $c \in \mathbb{R}$, then $f = 0$.
\end{proposition}
\begin{proof}
Since $(-\Delta)^{\frac{\alpha}{2}}$ commutes with translations, substituting $u(x,t)$ into \eqref{paley-wiener}, we obtain
\begin{equation*}
    -c f + f^m - (-\Delta)^{\frac{\alpha}{2}} (f^n) = 0,
\end{equation*}
where the constant of integration is zero by the compactness of the support of $f$ and the absolute integrability. Taking the Fourier transform, we have
\begin{equation}\label{fourier}
    |\xi|^\alpha \mathcal{F}[f^n](\xi) = \mathcal{F}[-cf + f^m].
\end{equation}
Since the Fourier transform of compactly supported distributions is entire by the Paley-Wiener theorem, \eqref{fourier} implies $|\xi|^\alpha$ has a meromorphic extension, a contradiction if $\alpha \in (0,2)$.    
\end{proof}
\section{Homotopy perturbation method}\label{method2}
\subsection{Method}\label{method}
The homotopy perturbation method (HPM) is briefly summarized with remarks on convergence/divergence. See \cite{he1999homotopy,he2000coupling} for an elaboration of HPM. Let $X,Y$ be Banach spaces that consist of functions defined on some domain $\Omega$ and let $L,N: X \rightarrow Y$ be continuous linear and nonlinear operators, respectively. Given a trial function $u_0 \in X$ and $f \in Y$, the main idea is to solve the boundary value problem
\begin{equation}\label{bvp}
\begin{split}
L u + N(u) &= f(\mathbf{r}),\ \mathbf{r} \in \Omega,\\
B\left(u,\frac{\partial u}{\partial n}\right) &= 0,\ \mathbf{r} \in \partial \Omega,    
\end{split}
\end{equation}
by defining a homotopy $H: X \times [0,1] \rightarrow Y$ given by
\begin{equation*}
    H(v,p) = (1-p)(Lv - Lu_0) + p (Lv + N(v) - f),
\end{equation*}
where we remark that there exist alternative constructions of homotopies \cite{momani2007homotopy,odibat2008modified}. Observe that $Lv - Lu_0$ and $Lv + N(v) - f$ are homotopic. Considering the zero set $\{H(v,p) = 0\}$, we wish to define a map $p \mapsto v(p)$ such that $H(v(p),p) = 0$. For $p=0$, take $v(0) = u_0$. Considering $p$ as the perturbation parameter, take the series ansatz
\begin{equation}\label{ansatz_powerseries}
    v(p) = \sum_{n=0}^\infty u_n p^n.
\end{equation}
Substitute \eqref{ansatz_powerseries} into $H(v,p) = 0$, and match coefficients of $p^n,\ n=0,1,\dots,$ to determine $u_n \in X$. It is crucial to show that the radius of convergence, $r_{conv}$, of the power series is at least $1$. If $r_{conv} > 1$, it follows from standard analysis that $u := v(1) = \sum\limits_{n=0}^{\infty} u_n \in X$ is absolutely convergent and is a solution of \eqref{bvp}; if $r_{conv} = 1$, then it needs to be shown that $\lim\limits_{p \rightarrow 1-} v(p)$ exists and is a solution of \eqref{bvp}.

Several other analytical methods have been developed to obtain exact series solutions to FDEs, including the variational iteration method (VIM) \cite{he2007variational}, the Adomian decomposition method (ADM) \cite{adomian2013solving}, and the homotopy analysis method (HAM) \cite{liao2003beyond}. While VIM effectively refines approximate solutions through a correction functional, it requires constructing appropriate Lagrange multipliers, which can be computationally challenging for certain fractional models; ADM, known for its straightforward recursive decomposition of nonlinear terms, may incur computational difficulties in explicitly determining the Adomian polynomials; HAM, though a powerful generalization of homotopy-based approaches, introduces an auxiliary parameter that requires careful tuning to ensure convergence, potentially adding to computational complexity.
\subsection{Divergence of HPM}
While the convergence analysis of \eqref{ansatz_powerseries} under various conditions were studied in \cite{he1999homotopy,he2000coupling,biazar2009convergence,elbeleze2014note}, we show a limitation of the method by giving an example that illustrates numerical divergence of HPM. Consider the time-fractional KdV equation (fKdV) with the one-soliton solution (of the KdV equation) traveling at the unit speed as a trial function. More precisely, let $\beta \in (0,1)$ and consider
\begin{equation*}
    \partial_t^\beta u + \partial_x^3 u + \partial_x \left(\frac{u^2}{2}\right) = 0,\ u_0 = 3 \sech^2\left(\frac{x}{2}\right).
\end{equation*}
Applying HPM with $L v = \partial_t^\beta v,\ N(v) = \partial_x^3 v + \partial_x \left(\frac{v^2}{2}\right)$ and matching coefficients of $p^n$, we have
\begin{equation*}
\begin{split}
    \partial_t^\beta u_0 &= 0,\ u_0(x,0) = 3 \sech^2 \left(\frac{x}{2}\right),\\
    \partial_t^\beta u_{n+1} &= -\partial_x^3 u_n - \frac{1}{2} \partial_x w_n,\ u_{n+1}(x,0) = 0,\ n = 0,1,\dots,\\
    w_n &= \sum\limits_{k=0}^{n} u_k u_{n-k}.
\end{split}
\end{equation*}
Integrating in $t$ by \eqref{fde}, the coefficients are given by
\begin{equation*}
\begin{split}
    u_0(x,t) &= 3 \sech^2 \left(\frac{x}{2}\right),\\
    u_{n+1}(x,t) &= -\frac{1}{\Gamma(\beta)} \int_0^t \left(\partial_x^3 u_n(x,\tau) + \frac{1}{2} \partial_x w_n (x,\tau) \right)(t-\tau)^{\beta - 1} d\tau.    
\end{split}
\end{equation*}
Few iterates are computed explicitly using Mathematica. Since the exact expression of $u_n$ for $n > 3$ becomes increasingly complicated, they are omitted for the sake of clarity.
\begin{figure}[h!]
    \centering
    \begin{subfigure}[b]{0.45\textwidth}
        \centering
        \includegraphics[width=\textwidth]{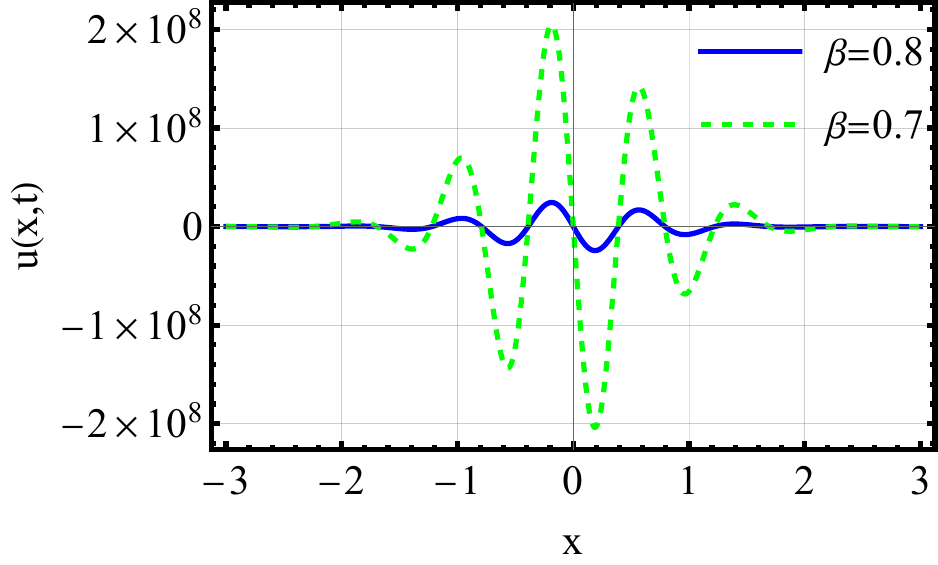}        
    \end{subfigure}
    \hfill
    \begin{subfigure}[b]{0.45\textwidth}
        \centering
        \includegraphics[width=\textwidth]{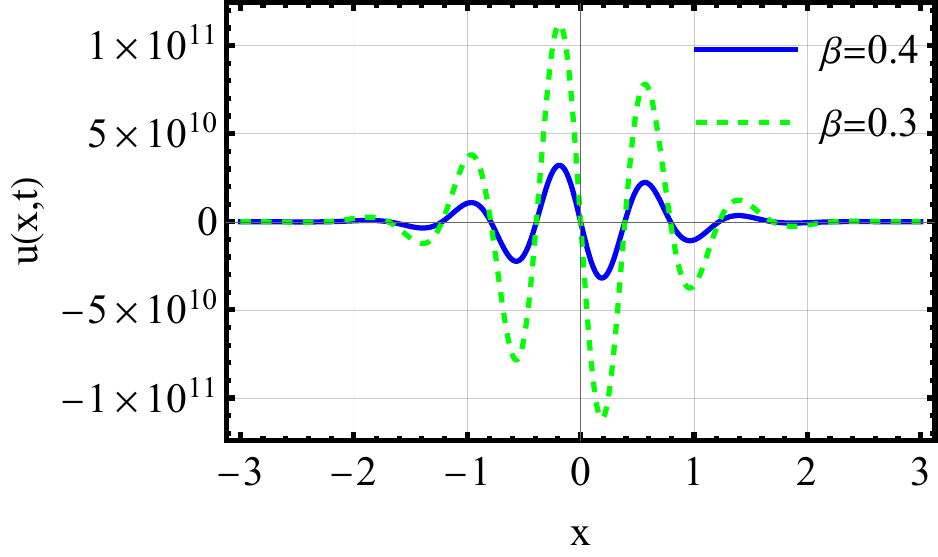}        
    \end{subfigure}
    \caption{Wave profiles approximated by $\sum\limits_{n=0}^{9} u_n(x,t)$ for multiple values of $\beta$ at $t=1$.}
    \label{fig:side_by_side}
\end{figure}
\begin{equation*}
\begin{split}
    u_1(x, t) &= \frac{24 \, t^\beta \, \text{csch}^3(x) \, \sinh^4\left(\frac{x}{2}\right)}{\Gamma(1 + \beta)},\\
    u_2(x, t) &= \frac{3 \, t^{2\beta} \, (-2 + \cosh(x)) \, \text{sech}^4\left(\frac{x}{2}\right)}{2 \, \Gamma(1 + 2\beta)},\\
    u_3(x, t) &= \frac{3 \, t^{3\beta} \, 
    \left[ (39 - 32 \cosh(x) + \cosh(2x)) \, \Gamma(1 + \beta)^2 + 
    12 \, (-2 + \cosh(x)) \, \Gamma(1 + 2\beta) \right] \, 
    \text{sech}^6\left(\frac{x}{2}\right) \, \tanh\left(\frac{x}{2}\right)}
    {8 \, \Gamma(1 + \beta)^2 \, \Gamma(1 + 3\beta)}.    
\end{split}   
\end{equation*}
\begin{figure}[H]
    \centering
    \begin{subfigure}[b]{0.45\textwidth}
        \centering
        \includegraphics[width=\textwidth]{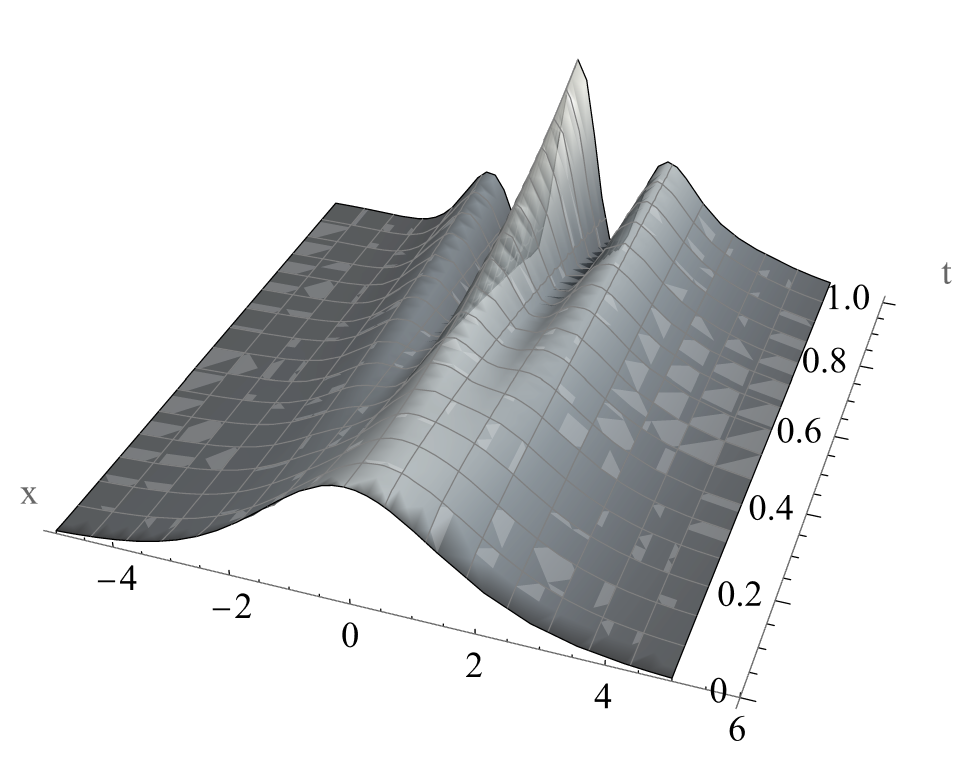}        
    \end{subfigure}
    \hfill
    \begin{subfigure}[b]{0.45\textwidth}
        \centering
        \includegraphics[width=\textwidth]{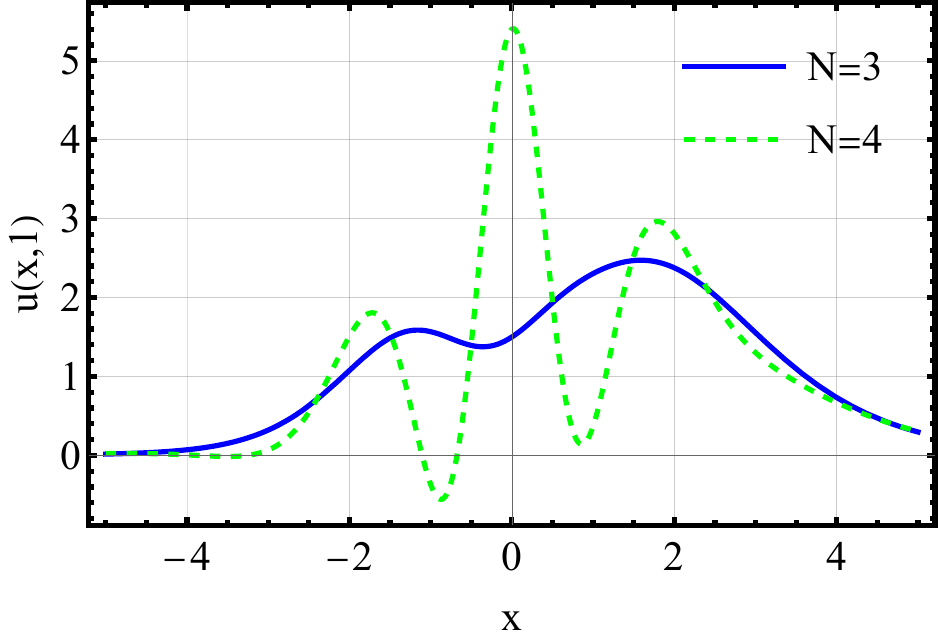}        
    \end{subfigure}
    \caption{Left: 3D plot of $\sum\limits_{n=0}^{4} u_n(x,t)$. Right: wave profiles approximated by $\sum\limits_{n=0}^{N} u_n(x,1)$ for $N=3,4$. Both plots use $\beta = 0.5$.}
    \label{fig:bifurcation}
\end{figure}
The solution profiles approximated by $\sum\limits_{n=0}^{N} u_n(x,1)$ for $N=9$ and multiple values of $\beta$ are given in \Cref{fig:side_by_side}. Note the increase of magnitude from $\sim 1$ at $t=0$ to $\sim 10^8, 10^{11}$ at $t=1$ in the figure. A 3D plot for $N=4$ and the solution profiles for $N=3,4$ at $t=1$ are given in \Cref{fig:bifurcation}. Observe the emergence of $2$ peaks for $N=3$ and $3$ peaks for $N=4$. This vastly distinct behavior of the solutions for multiple $\beta$ and $N$ raises doubts on the convergence of HPM.
\begin{table}[H]
    \centering    
    \begin{tabular}{lcccc|cccc}
        \toprule
        & \multicolumn{4}{c|}{$\mathbf{\| u_6 \| / \| u_5 \|}$} & \multicolumn{4}{c}{$\mathbf{\| u_7 \| / \| u_6 \|}$} \\ 
        \cmidrule(lr){2-5} \cmidrule(lr){6-9}
        $\mathbf{t}$ & \( \beta = 0.2 \) & \( \beta = 0.4 \) & \( \beta = 0.6 \) & \( \beta = 0.8 \) 
        & \( \beta = 0.2 \) & \( \beta = 0.4 \) & \( \beta = 0.6 \) & \( \beta = 0.8 \) \\ \midrule
        0.01 & 34.98 & 10.31 & 2.76 & 0.70 & 53.42 & 15.23 & 3.92 & 0.94 \\
       
        0.21 & 64.31 & 34.85 & 17.12 & 7.94 & 98.21 & 51.48 & 24.38 & 10.79 \\
        
        0.41 & 73.51 & 45.55 & 25.58 & 13.56 & 112.27 & 67.27 & 36.42 & 18.43 \\
        
        0.61 & 79.59 & 53.39 & 32.47 & 18.64 & 121.55 & 78.86 & 46.22 & 25.33 \\
        
        0.81 & 84.24 & 59.81 & 38.49 & 23.38 & 128.64 & 88.33 & 54.79 & 31.78 \\ \midrule
         
        & \multicolumn{4}{c|}{$\mathbf{\| u_8 \| / \| u_7 \|}$} & \multicolumn{4}{c}{$\mathbf{\| u_9 \| / \| u_8 \|}$} \\ 
        \cmidrule(lr){2-5} \cmidrule(lr){6-9}
        $\mathbf{t}$ & \( \beta = 0.2 \) & \( \beta = 0.4 \) & \( \beta = 0.6 \) & \( \beta = 0.8 \) 
        & \( \beta = 0.2 \) & \( \beta = 0.4 \) & \( \beta = 0.6 \) & \( \beta = 0.8 \) \\ \midrule
        0.01 & 94.71 & 26.26 & 6.57 & 1.54 & 127.74 & 34.57 & 8.44 & 1.93 \\
       
        0.21 & 174.12 & 88.76 & 40.85 & 17.58 & 234.84 & 116.84 & 52.45 & 22.01 \\
        
        0.41 & 199.04 & 116.00 & 61.03 & 30.02 & 268.46 & 152.69 & 78.36 & 37.59 \\
       
        0.61 & 215.51 & 135.98 & 77.46 & 41.25 & 290.66 & 178.99 & 99.45 & 51.65 \\
       
        0.81 & 228.08 & 152.32 & 91.83 & 51.75 & 307.62 & 200.49 & 117.90 & 64.81 \\ \bottomrule
         
    \end{tabular}
    \caption{The evolution of ratios $\frac{\| u_{n+1} \|}{\| u_n \|}$ under varying $\beta$ where $\| \cdot \|$ denotes the sup norm in $x \in \mathbb{R}$.}
    \label{tab:four_blocks}
\end{table}
If $\beta = 1$, then it can be directly verified that HPM recovers the KdV one-soliton, or i.e.,
\begin{equation}\label{hpm_kdv}
    3 \sech^2 \left(\frac{x-t}{2}\right) = \sum_{n=0}^{\infty} u_n(x,t).
\end{equation}
However for any $\beta \in (0,1)$, our numerical computation suggests the series diverges. \Cref{tab:four_blocks} contains four datasets of the ratios of norms where $\| u_n \|$ denotes the sup norm in $x \in \mathbb{R}$ for multiple values of $t,\beta$. If HPM were convergent, then the ratios are expected to decay in $n$ below $1$ by the ratio test. However our numerical computation suggests the ratios blow up monotonically for any $t>0,\ \beta \in (0,1)$, indicating divergence of $\sum\limits_{n=0}^\infty u_n$. As observed, the ratios blow up faster for lower values of $\beta$; HPM provides a convergent method for $\beta = 1$ by \eqref{hpm_kdv}. Similarly the ratios are larger for $t \gg 1$ since this application of HPM used $u_0$ as the initial condition of the KdV one-soliton. 

Our result needs to be compared to the previous results in \cite{wang2007homotopy,momani2005explicit}. The validity of HPM was claimed in \cite{wang2007homotopy} in the context of fKdV. Numerical solutions of fKdV were given in \cite{momani2005explicit} via the Adomian Decomposition Method (ADM), suggesting a bifurcation caused by the Caputo derivative, as a one-soliton breaks into multiple localized structures as $t$ evolves, similar to the right plot of \Cref{fig:bifurcation}. However both papers did not show convergence explicitly. Considering that HPM and ADM are both flexible methods that apply to a wide class of differential equations, more investigation is needed in obtaining accurate numerical solutions to fKdV.
\section{Fractional K(2,2)}\label{fk22}
In this section, HPM is applied to the fractional $K(2,2)$ equation. For \eqref{k22}, let 
\begin{equation*}
Lu = \partial_t^\beta u;\ N(u) = \partial_x (u^2) - \partial_x (-\partial_{xx})^{\frac{\alpha}{2}} (u^2),    
\end{equation*}
and
\begin{equation}\label{ansatz1}
    u_0 = \frac{4c}{3} \cos^2\left(\frac{x}{4}\right).
\end{equation}
It is known that $\frac{4c}{3} \cos^2\left(\frac{x-ct}{4}\right)$ is a traveling wave solution to the classical $K(2,2)$. Applying HPM, we have
\begin{equation*}
\partial_t^\beta v + p \left(\partial_x (v^2) - \partial_x (-\partial_{xx})^{\frac{\alpha}{2}} (v^2) \right) = 0.
\end{equation*}
The series ansatz \eqref{ansatz_powerseries} yields a recurrence relation
\begin{equation}\label{recurrence_k22}
\partial_t^\beta u_{n+1} = -\partial_x w_n + \partial_x (-\partial_{xx})^{\frac{\alpha}{2}} w_n,\ u_{n+1}(x,0) = 0,\ n = 0,1,\dots,   
\end{equation}
where $w_n = \sum\limits_{k=0}^{n} u_k u_{n-k}$. Integrating \eqref{recurrence_k22} in $t$ by \eqref{fde} and using $u_0$, all higher coefficients are computed explicitly as
\begin{equation}\label{k22_un}
\begin{split}
u_n(x,t) &= (-1)^{\lfloor \frac{n}{2} \rfloor} \frac{2^{n+1-n\alpha}(2^\alpha - 1)^n c^{n+1} t^{\beta n}}{3^{n+1} \Gamma(1+\beta n)} \psi\left(\frac{x}{2}\right),\ n \geq 1,\\
\psi(x) &=
\begin{cases}
\cos (x), & \text{if n is even},\\
\sin (x), & \text{if n is odd}.
\end{cases}
\end{split}
\end{equation}
The magnitudes of $u_n$ decay sufficiently fast as $n$ increases, and it can be verified that the series is absolutely convergent at $p=1$. Let $\gamma = \frac{2^{1-\alpha} (2^\alpha - 1)}{3}$. Then,
\begin{equation}\label{k22_solution}
u(x,t) = \sum_{n=0}^\infty u_n(x,t)=\frac{2c}{3} \biggl\{ 1 + 
    \cos\left(\frac{x}{2}\right) 
    E_{2\beta}\left(- (\gamma c t^\beta)^2 \right) + 
    \sin\left(\frac{x}{2}\right)\gamma c t^\beta
    E_{2\beta, 1+\beta}\left(-(\gamma c t^\beta)^2 \right)\biggr\}
\end{equation}
for any $x,t \in \mathbb{R}$. If $\alpha = 2,\ \beta = 1$, then \eqref{k22_solution} $= \frac{4c}{3} \cos^2\left(\frac{x-ct}{4}\right)$, the traveling wave solution for the classical $K(2,2)$. Note that the solution is stationary if $\alpha = 0$, which is consistent by formally substituting $\alpha = 0$ to \eqref{k22}. Without loss of generality, assume $c = 1,\ \alpha = 2$ by scaling the time variable as $T = (\gamma c)^{\frac{1}{\beta}}t$. Hence the Caputo derivative in time plays a pivotal role in the dynamics of the fractional $K(2,2)$.
\subsection{Dynamics of fractional $K(2,2)$}\label{dynamics_k22}
We investigate the role of $\beta$ in the dynamics of the fractional $K(2,2)$ equation. By elementary trigonometry, we have
\begin{equation}\label{def}
\begin{split}
\eqref{k22_solution} &= \frac{2}{3} \left(1 + A(t) \cos\left( \frac{x - \phi(t)}{2}\right)\right),\\
A(t) &:= \left(E_{2\beta}^2\left(-\frac{t^{2\beta}}{4}\right) + \frac{t^{2\beta}}{4}E_{2\beta,1+\beta}^2\left(-\frac{t^{2\beta}}{4}\right)\right)^{1/2},\\
\tan\left(\frac{\phi(t)}{2}\right) &:= \frac{t^\beta E_{2\beta,1+\beta}\left(-\frac{t^{2\beta}}{4}\right)}{2 E_{2\beta}\left(-\frac{t^{2\beta}}{4}\right)},
\end{split}
\end{equation}
where we define $\phi$ taking values in $\mathbb{R}$ with the condition $\phi(0) = 0$ to determine the phase uniquely. Observe that the solution is spatially periodic for any $t$. By direct computation, $A(t) = 1,\ \phi(t) = t$ if $\beta = 1$, consistent with the exact solution (traveling wave) of the classical $K(2,2)$ equation.
\subsubsection{Decay of amplitude}
We discuss the role of $\beta$ in the decay of $A(t)$ with short-time asymptotics and global decay estimates, assisted by numerical results.
\begin{proposition}
For $t \ll 1$ and $\beta \in (0,1)$, $A(t)$ decreases. More precisely, as $t \rightarrow 0+$,
    \begin{equation}\label{asymp_short}
        A(t) = 1 - \frac{t^{2\beta}}{8} \left(\frac{2}{\Gamma(1+2\beta) }- \frac{1}{\Gamma(1+\beta)^2}\right) + O(t^{4\beta}).
    \end{equation}    
Furthermore, there exist $C, t_0 >0$ such that for any $t \geq t_0$,
\begin{equation}\label{long-time}
\left|u(x,t) - \frac{2}{3} \right| \leq \frac{C t^{-\beta}}{\Gamma(1-\beta)}.    
\end{equation}
\end{proposition}
\begin{proof}
By standard algebra and \eqref{ml_decay}, the asymptotics \eqref{asymp_short} follows, and \eqref{long-time} follows from
\begin{equation}\label{asymp_longtime}
A(t) = \frac{2}{\Gamma(1-\beta)}t^{-\beta} + O(t^{-2\beta}),\ t \rightarrow \infty.
\end{equation}
It can be shown explicitly that $\frac{2}{\Gamma(1+2\beta) }- \frac{1}{\Gamma(1+\beta)^2} \geq 0$ with the equality if and only if $\beta = 1$. Hence for $t \ll 1$, $A(t)$ decreases for any $\beta < 1$. Furthermore $\lim\limits_{t \rightarrow 0+} A^\prime(t) = -\infty$ for $\beta \in (0,\frac{1}{2})$ whereas the limit converges to a finite value for $\beta \in [\frac{1}{2},1]$. The asymptotic formula \eqref{asymp_longtime} needs to be determined separately for $\beta = \frac{1}{2}$ since \eqref{ml_decay} contains the term $\Gamma(1-2\beta)^{-1}$. The result is true, nevertheless, since 
\begin{equation*}
A(t)^2 = e^{-\frac{t}{2}}\left(1+\text{erfi}\left(\frac{\sqrt{t}}{2}\right)^2\right)    
\end{equation*}
at $\beta = \frac{1}{2}$ where the imaginary error function is defined as
\begin{equation}\label{erfi}
\text{erfi}(x) := -i erf(ix) \sim \frac{e^{x^2}}{\sqrt{\pi}x}\left( 1 + \sum_{n=1}^\infty \frac{(-1)^n (2n-1)!!}{(2x^2)^n}\right),\ x \rightarrow \infty,    
\end{equation}
in the sense of asymptotic expansion. Moreover \eqref{asymp_longtime} does not contradict the constancy of $A(t)$ at $\beta = 1$ due to the factor $\Gamma(1-\beta)^{-1}$.    
\end{proof}

\noindent \textbf{Transition to long time:} There exists $\beta_c \approx 0.672$ such that $A(t)$ decays to zero monotonically if $\beta \in (0,\beta_c)$, and $A(t)$ decays to zero via finitely many oscillations if $\beta \in (\beta_c,1)$. See \Cref{fig:amplitude}.

If $A(t)$ does not decay to zero monotonically, then there exists $t_\beta > 0$ such that $A(t)$ obtains the first local minimum at $t = t_\beta$. We verified numerically that $A^\prime(t) = 0$ has a root if and only if $\beta > \beta_c \approx 0.672$. For such $\beta
$, there cannot be infinitely many critical points of $A(t)$, since this contradicts the existence of finitely many positive zeros of $E_{2\beta}(-x)$. \Cref{fig:amplitude} gives the log of amplitudes for multiple $\beta$. If $\beta > \beta_c$, \Cref{fig:plot1} plots the first local minimum, emphasized with the dots at the peaks. If $\beta < \beta_c$, no peaks exist. \Cref{fig:plot3} is highly suggestive of the discontinuity of $t_\beta$ in $\beta$. As $\beta$ increases, the smallest local minimum \textit{plateaus}, and the critical point \textit{jumps} to a higher value of $t$. In \Cref{tab:beta_table}, we observed numerically that these jumps occur indefinitely, leading to $\lim\limits_{\beta \rightarrow 1-} t_{\beta} = \infty$.

\begin{figure}[ht!]
    \centering
    \begin{subfigure}[t]{0.32\textwidth}
        \centering
        \includegraphics[width=\textwidth]{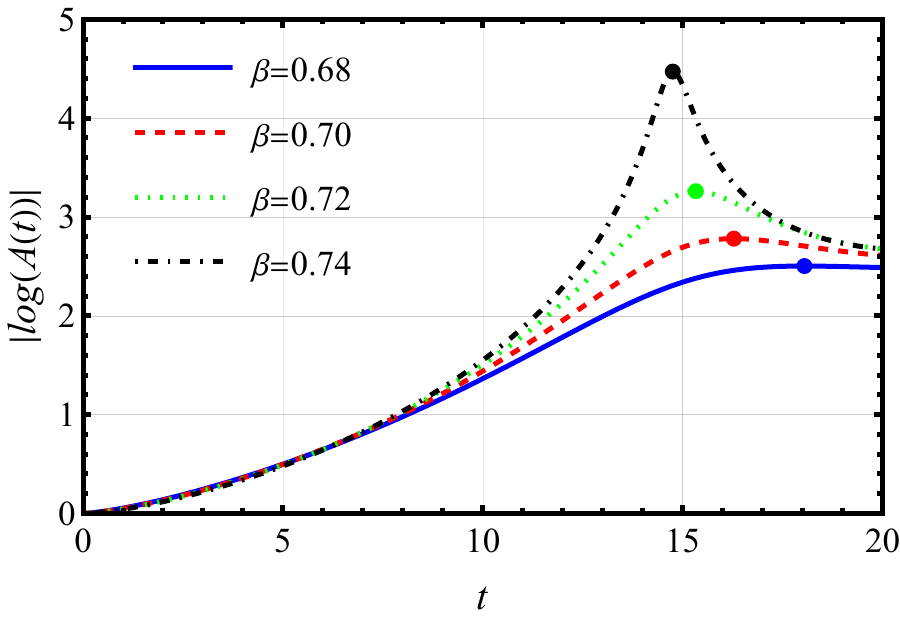} 
        \caption{$\beta > \beta_c$}
        \label{fig:plot1}
    \end{subfigure}
    \hfill
    \begin{subfigure}[t]{0.331\textwidth}
        \centering
        \includegraphics[width=\textwidth]{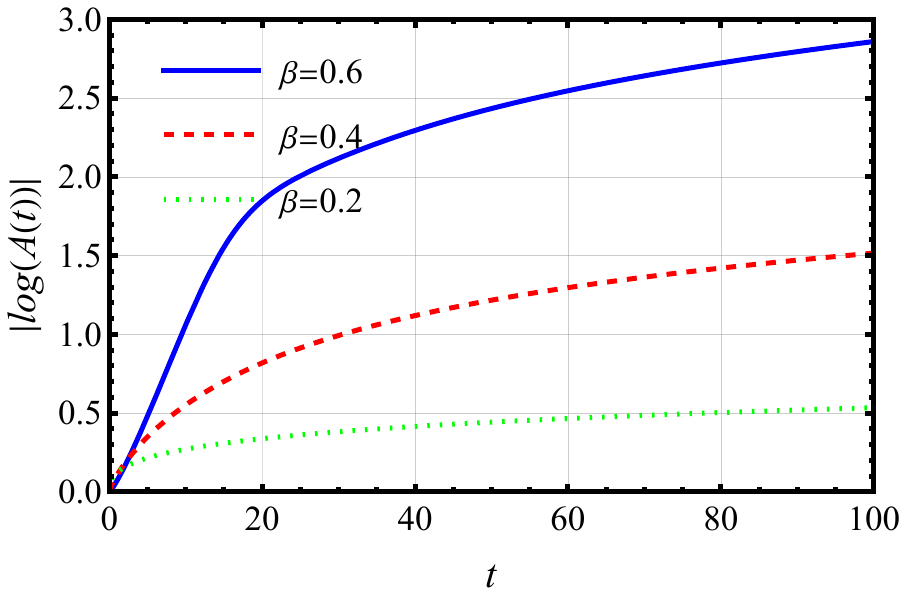} 
        \caption{$\beta < \beta_c$}
        \label{fig:plot2}
    \end{subfigure}
    \hfill
    \begin{subfigure}[t]{0.32\textwidth}
        \centering
        \includegraphics[width=\textwidth]{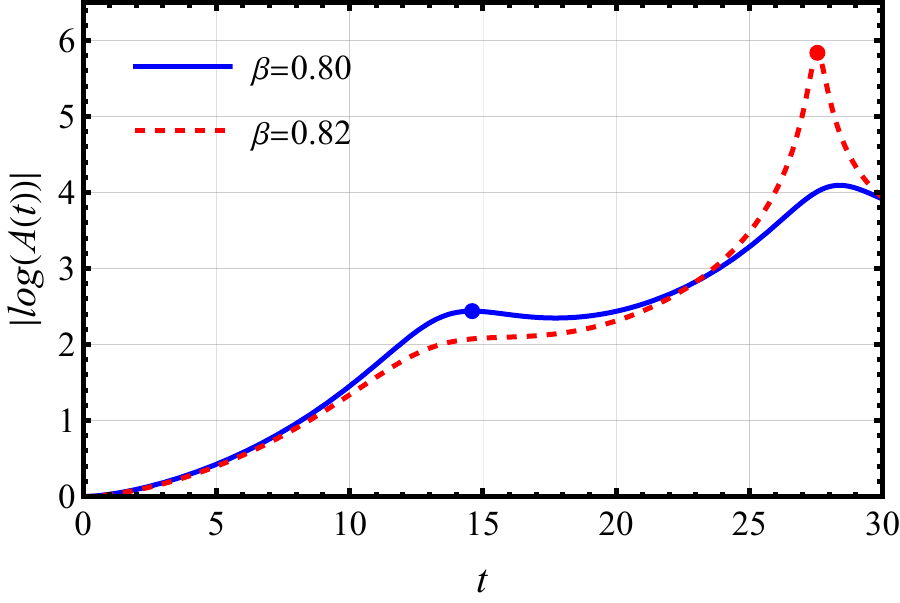} 
        \caption{Jump in $t_\beta$}
        \label{fig:plot3}
    \end{subfigure}    
    \caption{Evolution of $A(t)$: whether $A(t)$ decays to zero monotonically depends on $\beta$ relative to $\beta_c$.}
    \label{fig:amplitude}
\end{figure}

\begin{table}[ht!]
    \centering    
    \begin{tabular}{lcccccccccc} 
        \toprule
        $\beta$ & 0.80 & 0.82 & 0.84 & 0.86 & 0.88 & 0.90 & 0.92 & 0.94 & 0.96 & 0.98 \\ 
        $t_{\beta}$ & 14.60 & 27.56 & 27.13 & 27.23 & 39.83 & 40.71 & 65.28 & 90.81 & 154.07 & 356.19 \\ 
        \bottomrule
    \end{tabular}
    \caption{$\beta > \beta_c$: the first local minimum of $A(t)$.}
    \label{tab:beta_table}
\end{table}

\subsubsection{Phase-locking}
An elaboration of \Cref{phase_evolve} is given, followed by proofs. Unlike the case $\beta = 1$ where $\phi(t) = t$, and hence no phase-locking, the phase for $\beta < 1$ locks to an odd integer multiple of $\pi$ for $t \gg 1$.  Note that $\frac{\Gamma(1-\beta)}{\Gamma(1-2\beta)}$ is positive if $\beta \in (0,\frac{1}{2})$ and negative if $\beta \in (\frac{1}{2},1)$. Therefore $\phi$ approaches $\pi$ from the left (undershoot) in \eqref{phase1}, \eqref{phase2} and approaches $(2k-1)\pi$ from the right (overshoot) in \eqref{phase3}. \Cref{fig:undershoot} illustrates the monotonic growth for $\beta \leq \frac{1}{2}$ that undershoots to $\pi$. The critical threshold is $\beta = \frac{1}{2}$ where the difference to $\pi$ is exponentially small in $t$ as opposed to algebraic smallness for $\beta < \frac{1}{2}$. On the other hand, overshoot occurs if the value of $\tan\left(\frac{\phi(t)}{2}\right)$ transitions from $+\infty$ to $-\infty$ or vice versa, which occurs if and only if the denominator $E_{2\beta}(-\frac{t^{2\beta}}{4})$ takes non-degenerate (non-tangential) zeros; if $\beta \in (0,\frac{1}{2}]$, the denominator is never zero for positive $t$ due to complete monotonicity.
\begin{figure}[ht]
    \centering
    \begin{subfigure}[b]{0.45\textwidth}
        \centering
        \includegraphics[width=\textwidth]{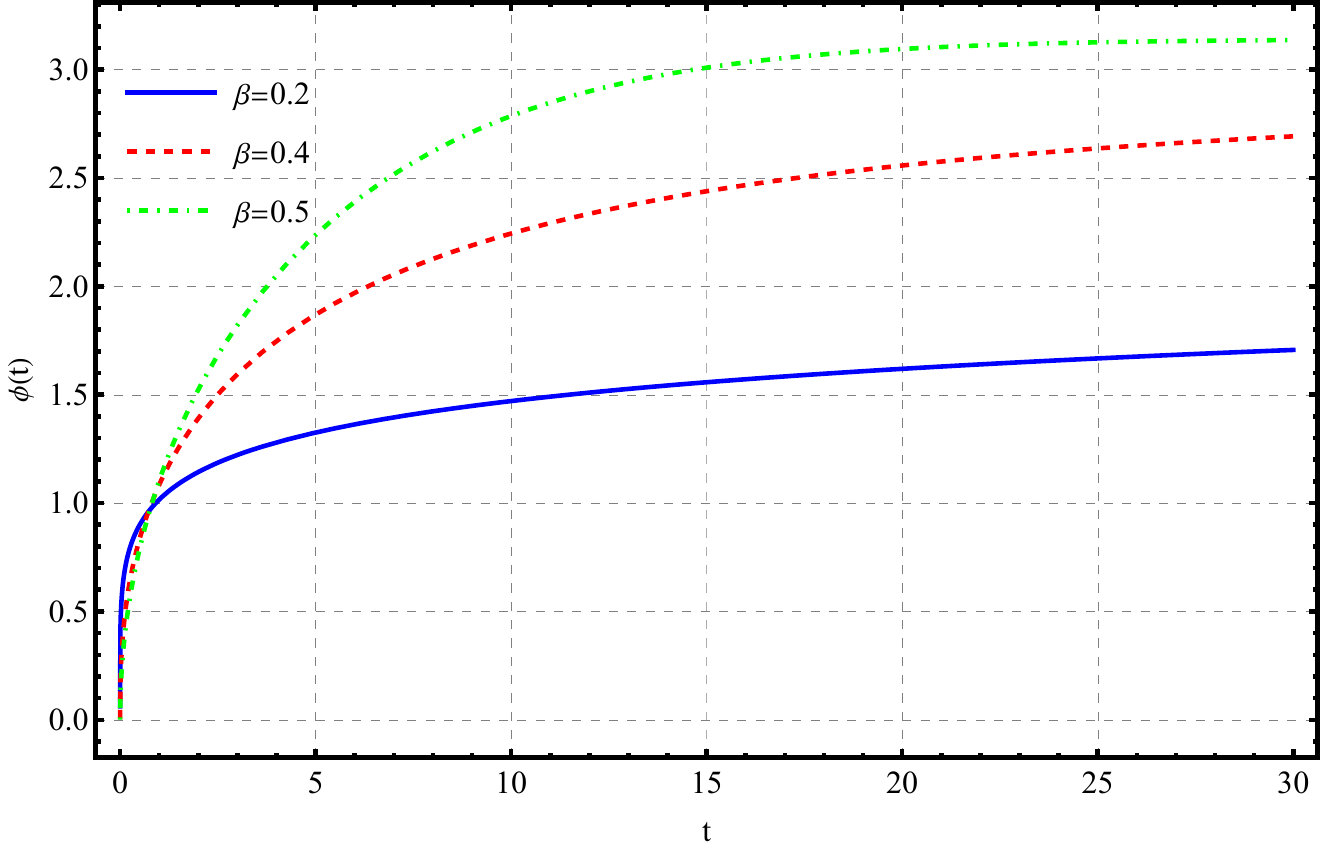}
    \end{subfigure}
    \hfill
    \begin{subfigure}[b]{0.455\textwidth}
        \centering
        \includegraphics[width=\textwidth]{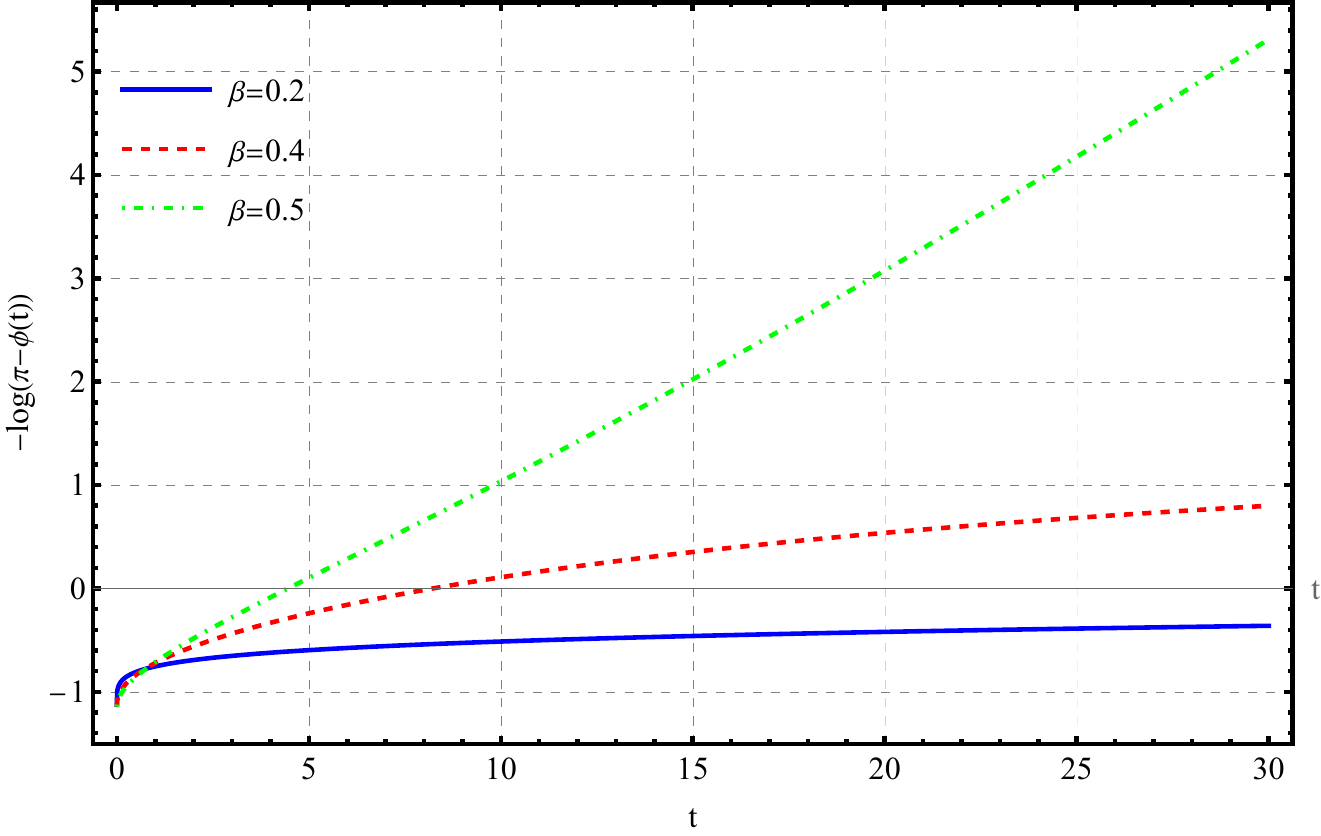}
    \end{subfigure}
    \caption{Left: monotonic growth of $\phi(t)$ for $\beta \leq \frac{1}{2}$. Right: Growth of $|\log (\pi - \phi(t))|$.}
    \label{fig:undershoot}
\end{figure}
\begin{proposition}\label{phase_evolve}
For $t \ll 1$, we have
    \begin{equation}\label{phase_short}
    \phi(t) = \frac{t^\beta}{\Gamma(1+\beta)} + O(t^{3\beta}),\ t \rightarrow 0.    
    \end{equation}
Furthermore, the transition to long-time dynamics depends on $\beta$. More precisely,

    \begin{enumerate}
        \item If $\beta \in (0,\frac{1}{2})$, then $\phi(t)$ increases to $\pi$ monotonically as $t$ increases, and
        \begin{equation}\label{phase1}
            \phi(t) = \pi - \frac{4 \Gamma(1-\beta)}{\Gamma(1-2\beta)}t^{-\beta} + O(t^{-3\beta}),\ t \rightarrow \infty.
        \end{equation}
    \item If $\beta = \frac{1}{2}$, then $\phi(t)$ increases to $\pi$ monotonically as $t$ increases, and
        \begin{equation}\label{phase2}
            \phi(t) = \pi - \sqrt{\pi} e^{-\frac{t}{4}}t^{\frac{1}{2}} + O(e^{-\frac{3t}{4}}t^{\frac{3}{2}}),\ t \rightarrow \infty.
        \end{equation}
    \item If $\beta \in (\frac{1}{2},1)$, then there exists $k = k(\beta) \in \mathbb{Z}$ such that
        \begin{equation}\label{phase3}
            \phi(t) = (2k-1)\pi + \left(\frac{-4 \Gamma(1-\beta)}{\Gamma(1-2\beta)} \right)t^{-\beta} + O(t^{-3\beta}),\ t \rightarrow \infty,
        \end{equation}
        where $|k|$ is bounded above by the number of roots of $E_{2\beta}(-x)$ where $x > 0$.
    \end{enumerate}
\end{proposition}
\begin{figure}[ht]
    \centering
    \includegraphics[width=0.45\textwidth]{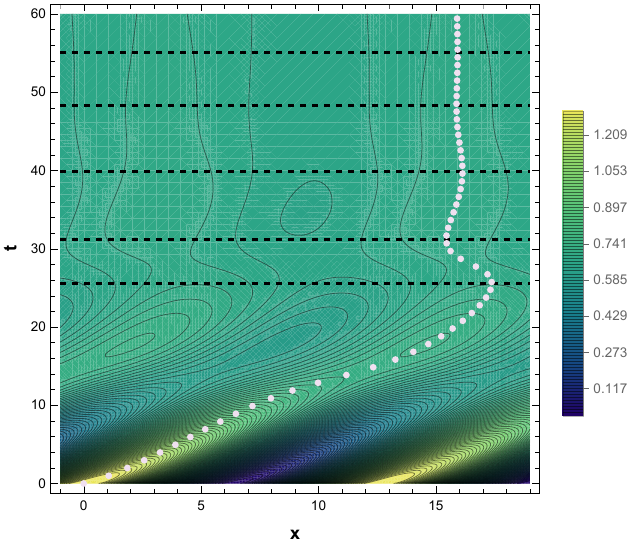} 
    \hfill 
    \includegraphics[width=0.45\textwidth]{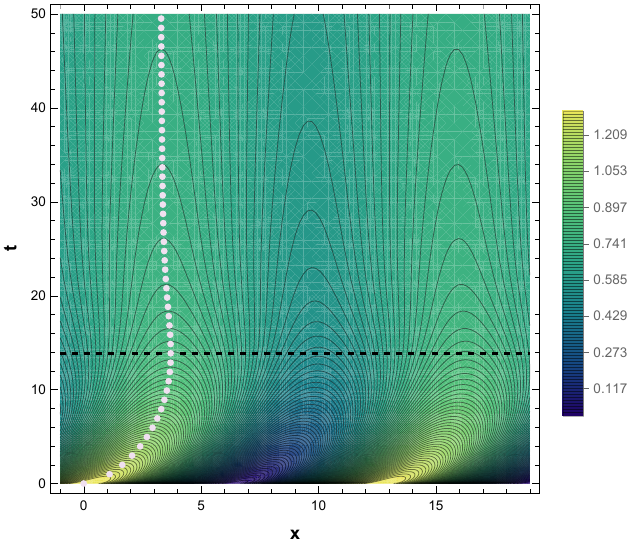} 
    \caption{Contour plots of $u(x,t)$ with $\beta = 0.8$ (Left) and $\beta = 0.6$ (Right) with $\alpha = 2,\ c = 1$. The evolution of phase $x = \phi(t)$ is marked by the round dots. The sinusoidal waves change the direction at times marked by the dotted horizontal lines.}
    \label{fig:contour}
\end{figure}
\begin{proof}
The short-time behavior \eqref{phase_short} follows from the definition of $\phi$ and the Mittag-Leffler function. For the long-time behavior, \eqref{phase1} follows from \eqref{ml_decay}. The monotonic increasing behavior of \eqref{phase1} is a consequence of the complete monotonicity of $E_{2\beta}(-x^{2\beta})$ for $x > 0$ and $2\beta \in (0,1)$. If $\beta = \frac{1}{2}$, the exact expression for $\phi(t)$ reduces to a special function. More precisely, $\tan \left(\frac{\phi(t)}{2}\right) = \text{erfi}\left(\frac{\sqrt{t}}{2}\right)$. The monotonic increasing behavior and \eqref{phase2} follow from the properties of the imaginary error function and the asymptotic expansion \eqref{erfi}.

To show the last claim \eqref{phase3}, let $\theta(t) = \frac{\phi(t)}{2}$. As $\tan \theta(t)$ suffers from the discontinuity transitioning from $+\infty$ to $-\infty$ or vice versa, $\theta(t)$ evolves continuously due to the periodicity of $\tan(\cdot)$. To make this precise, let $E_j = [(j-\frac{1}{2})\pi,(j+\frac{1}{2})\pi) \subseteq \mathbb{R}$ for $j \in \mathbb{Z}$. If $\beta \in (\frac{1}{2},1)$, then $E_{2\beta}(-x)$ has finitely many positive zeros. Since $E_{2\beta}(-x) \sim \frac{x^{-1}}{\Gamma(1-2\beta)}$ as $x \rightarrow \infty,\ E_{2\beta}(0) = 1$, and $\Gamma(1-2\beta) < 0$, there exists a non-degenerate zero by the Intermediate Value Theorem. The denominator function $E_{2\beta}(-\frac{t^{2\beta}}{4})$ switches sign from positive to negative at its first non-degenerate zero. Therefore $\theta$ moves from $E_0$ to $E_1$. Likewise every non-degenerate zero corresponds to a transition from $E_{j_0}$ to $E_{j_0 \pm 1}$, and there exists at most $N_0$ transitions where $N_0(\beta)$ is the number of zeros of $E_{2\beta}(-x)$. Take $T \gg 1$ such that $\tan \theta(t) \approx \frac{\Gamma(1-2\beta)t^\beta}{2\Gamma(1-\beta)}$ for $T \geq t$ and $T$ is greater than the largest positive root of $E_{2\beta}(-\frac{t^{2\beta}}{4})$. For any $t \geq T$, $\theta(t) \in E_{k}$ for some $k \in \mathbb{Z}$ where $|k| \leq N_0$. By expanding the inverse tangent at the infinity, we obtain \eqref{phase3} as
\begin{equation*}
    \theta(t) = k\pi - \frac{\pi}{2} -\frac{2 \Gamma(1-\beta)}{\Gamma(1-2\beta)} t^{-\beta} + O(t^{-3\beta}),\ t \rightarrow \infty.
\end{equation*}
\end{proof}

\Cref{fig:contour} illustrates the long-time behavior that overshoots to $(2k-1)\pi$ where $k(0.8) = 3,\ k(0.6) = 1$. The horizontal lines denote the critical points of $\phi(t)$ given by $\{25.53,31.21,39.91,48.24,55.06\}$ for $\beta = 0.8$ and $\{13.87\}$ for $\beta = 0.6$. The round dots denote the evolution of phase $x = \phi(t)$. Given that the phase locks as $t \rightarrow \infty$ and the amplitude decays as $t^{-\beta}$, where the solution converges uniformly to a constant, we conclude that the solution obtained from HPM with the trial function \eqref{ansatz1} is not a traveling wave solution.

\section{Fractional K(3,3)}\label{fk33}
We continue to assume $\alpha = 2,\ c = 1$ and apply HPM to derive solutions to the fractional $K(3,3)$. Let
\begin{equation}\label{ansatz2}
    u_0 = \left(\frac{3c}{2}\right)^{\frac{1}{2}} \cos\left(\frac{x}{3}\right).
\end{equation}
It is known that $\left(\frac{3c}{2}\right)^{\frac{1}{2}} \cos\left(\frac{x-ct}{3}\right)$ is a traveling wave solution to the classical $K(3,3)$. Applying HPM, we have
\begin{equation*}
\partial_t^\beta v + p \left(\partial_x (v^3) + \partial_x^3(v^3) \right) = 0.
\end{equation*}
The series ansatz \eqref{ansatz_powerseries} yields a recurrence relation
\begin{equation}\label{recurrence_k33}
\begin{split}
\partial_t^\beta u_{n+1} &= -\partial_x w_n - \partial_x^3 w_n,\ u_{n+1}(x,0) = 0,\ n = 0,1,\dots,\\
w_n &= \sum_{n_1+n_2+n_3 = n} u_{n_1}u_{n_2}u_{n_3}.   
\end{split}
\end{equation}
Integrating \eqref{recurrence_k33} in $t$ and using $u_0$, all higher coefficients are computed explicitly as
\begin{equation}
\begin{split}
u_1(x,t) &= \frac{t^\beta \sin\left(\frac{x}{3}\right)}{\sqrt{6} \, \Gamma(1 + \beta)},\\
u_2(x,t) &= -\frac{t^{2\beta} \cos\left(\frac{x}{3}\right)}{3 \sqrt{6} \, \Gamma(1 + 2\beta)},\\
u_3(x,t) &= -\frac{t^{3\beta} \left(3 \Gamma(1 + \beta)^2 - \Gamma(1 + 2\beta)\right) \sin\left(\frac{x}{3}\right)}{9 \sqrt{6} \, \Gamma(1 + \beta)^2 \Gamma(1 + 3\beta)}.
\end{split}
\end{equation}
Unlike \eqref{k22_un}, the exact expressions for higher order terms have no evident pattern for $n > 3$. 
\begin{proposition}
There exists a real sequence $\{c_n(\beta)\}_{n=0}^\infty$ with
\begin{equation*}
c_n = 0,\ \text{for }n=0,1,2, \text{ and } c_3 = \frac{3\Gamma(1+\beta)^2 - \Gamma(1+2\beta)}{\Gamma(1+\beta)^2},    
\end{equation*}
such that for $n \geq 0$,
\begin{equation*}
u_n(x,t) = (-1)^{\lfloor \frac{n}{2} \rfloor} \frac{c_n(\beta) 2^{-\frac{1}{2}} 3^{-(n-\frac{1}{2})} t^{\beta n}}{\Gamma(1+\beta n)} \psi\left(\frac{x}{3}\right);\ \psi(x) =
\begin{cases}
\cos (x), & \text{if n is even},\\
\sin (x), & \text{if n is odd},
\end{cases}\\
\end{equation*}
and
\begin{equation}\label{recurrence}
\begin{split}
c_{n+1} &= (-1)^{\lfloor \frac{n+1}{2} \rfloor} \sum_{n_1+n_2+n_3 = n}  \frac{(-1)^{\lfloor \frac{n_1}{2} \rfloor + \lfloor \frac{n_2}{2} \rfloor + \lfloor \frac{n_3}{2} \rfloor} 3^{-\sigma} \Gamma(1+\beta n) c_{n_1}c_{n_2}c_{n_3}}{\Gamma(1+\beta n_1)\Gamma(1+\beta n_2)\Gamma(1+\beta n_3)},\\
\sigma(n_1,n_2,n_3) &:=
\begin{cases}
0, & \text{if $n_j$ is even for all $1 \leq j \leq 3$, or $n_j$ is odd for all $1 \leq j \leq 3$},\\
1, & \text{otherwise}.
\end{cases}
\end{split}
\end{equation}
Therefore,
\begin{equation*}
\begin{split}
u(x,t) &= \sum_{k=0}^\infty (-1)^k \frac{c_{2k}(\beta) 2^{-\frac{1}{2}}3^{-(2k-\frac{1}{2})}t^{2\beta k}}{\Gamma(1+2\beta k)}\cos\left(\frac{x}{3}\right) + \sum_{k=0}^\infty (-1)^k \frac{c_{2k+1}(\beta) 2^{-\frac{1}{2}}3^{-(2k+\frac{1}{2})}t^{\beta(2k+1)}}{\Gamma(1+\beta(2k+1))}\sin\left(\frac{x}{3}\right)\\
&=: E_c(t;\beta) \cos\left(\frac{x}{3}\right) + E_s(t;\beta) \sin\left(\frac{x}{3}\right).
\end{split}
\end{equation*}    
\end{proposition}
\begin{proof}
Let $n$ be even; the argument for $n$ odd is similar. Then $n = n_1 + n_2 + n_3$ holds for $n_j$ even for all $1 \leq j \leq 3$, or $n_j$ being even for some $j$ and the other two indices being odd. The operator $A := \partial_x + \partial_{x}^3$ acts on the spatial functions as
\begin{equation*}
\begin{split}
A[\cos^3\left(\frac{\cdot}{3}\right)](x) &= -\frac{2}{9} \sin \left(\frac{x}{3}\right),\\
\ A[\sin^2\left(\frac{\cdot}{3}\right)\cos\left(\frac{\cdot}{3}\right)](x) &= -\frac{2}{27} \sin \left(\frac{x}{3}\right).
\end{split}
\end{equation*}
By applying the fractional integral, we have
\begin{align*}
    u_{n+1} &= -\frac{1}{\Gamma(\beta)}\int_0^t A w_n (x,\tau) (t-\tau)^{\beta - 1} d \tau\\
    &= -\frac{t^{\beta (n+1)}\Gamma(1+\beta n)}{\Gamma(1+\beta (n+1))} \sum_{n_1+n_2+n_3 = n} \frac{(-1)^{\lfloor \frac{n_1}{2}\rfloor + \lfloor \frac{n_2}{2}\rfloor + \lfloor \frac{n_3}{2}\rfloor}3^{-(n-\frac{3}{2})}c_{n_1}c_{n_2}c_{n_3}}{2^{\frac{3}{2}}\Gamma(1+\beta n_1)\Gamma(1+\beta n_2)\Gamma(1+\beta n_3)}A[\psi_{n_1}\left(\frac{\cdot}{3}\right)\psi_{n_2}\left(\frac{\cdot}{3}\right)\psi_{n_3}\left(\frac{\cdot}{3}\right)]\\
    &= \frac{2^{-\frac{1}{2}}3^{-(n+\frac{1}{2})}t^{\beta(n+1)}}{\Gamma(1+\beta(n+1))} \sum_{n_1+n_2+n_3 = n} \frac{(-1)^{\lfloor \frac{n_1}{2}\rfloor + \lfloor \frac{n_2}{2}\rfloor + \lfloor \frac{n_3}{2}\rfloor}3^{-\sigma} \Gamma(1+\beta n) c_{n_1}c_{n_2}c_{n_3}}{\Gamma(1+\beta n_1)\Gamma(1+\beta n_2)\Gamma(1+\beta n_3)} \sin\left(\frac{x}{3}\right),
\end{align*}
from which \eqref{recurrence} follows.    
\end{proof}
It can be checked directly that $c_n(1) = 1$ for all $n \geq 0$. 
\subsection{Dynamics of fractional $K(3,3)$}
The asymptotics for the short-time behavior can be argued similarly as the fractional $K(2,2)$. We give a numerical evidence of finite time blow-up. Consequently, the traveling wave with the initial condition $u_0$ for the classical $K(3,3)$ is no longer a traveling wave for $\beta < 1$ via HPM. As \eqref{def}, define
\begin{equation}\label{def2}
A(t) = \left(E_c(t)^2+E_s(t)^2\right)^{\frac{1}{2}},\ \tan \left(\frac{\phi(t)}{3}\right) = \frac{E_s(t)}{E_c(t)},
\end{equation}
such that $u(x,t) = A(t) \cos\left(\frac{x-\phi(t)}{3}\right)$, and $\phi:(-T,T) \rightarrow \mathbb{R}$ is continuous and $\phi(0) = 0$.
\begin{proposition}
For any $\beta \in (0,1)$, we have
\begin{equation}\label{asymp_short2}
    \begin{split}
        A(t) &= \left(\frac{3}{2}\right)^{\frac{1}{2}}-\frac{1}{3\sqrt{6}} \left(\frac{1}{ \Gamma (1+2\beta)}-\frac{1}{2 \Gamma (1+\beta)^2}\right)t^{2\beta} + O\left(t^{4\beta}\right),\ t \rightarrow 0,\\
        \phi(t) &= \frac{t^\beta}{\Gamma(1+\beta)} + O(t^{3\beta}),\ t \rightarrow 0,
    \end{split}
\end{equation}
where $\frac{1}{ \Gamma (1+2\beta)}-\frac{1}{2 \Gamma (1+\beta)^2} \geq 0$ with the equality if and only if $\beta = 1$.    
\end{proposition}
\noindent\textbf{Growth of $c_n(\beta)$:} For any $\beta \in (0,1)$, there exists $\gamma(\beta) > 0$ such that $c_n^{1/n} \sim \gamma \Gamma(1+\beta n)^{\frac{1}{n}},\ n \rightarrow \infty$.

The asymptotics \eqref{asymp_short2} is immediate from the definition \eqref{def2}. Therefore the fractional $K(2,2)$ and $K(3,3)$ exhibit similar short-time behavior where the amplitude decreases and the phase $\sim \frac{t^\beta}{\Gamma(1+\beta)}$. The growth of $c_n(\beta)$ is estimated numerically to be super-exponential, implying that the radius of convergence of both $E_c,E_s$ is finite. Therefore the solution cannot be extended above some terminal time $T>0$, if smoothness in $t$ were desired.

\begin{figure}[ht]
    \centering
    \begin{minipage}[t]{0.45\textwidth} 
        \centering
        \vspace{7ex}
       \begin{tabular}{cc}
    \toprule
    $\beta$ & $R$ \\
    \midrule
    0.1 & 10.0508 \\
    0.3 & 5.1655 \\
    0.5 & 5.6875 \\
    0.7 & 7.0081 \\
    0.9 & 10.2901 \\
    \bottomrule
\end{tabular}
    \end{minipage}
    \hfill
    \begin{minipage}[t]{0.54\textwidth} 
        \centering
        \vspace{-28ex} 
         \includegraphics[width=\textwidth, trim=50 0 50 0, clip]{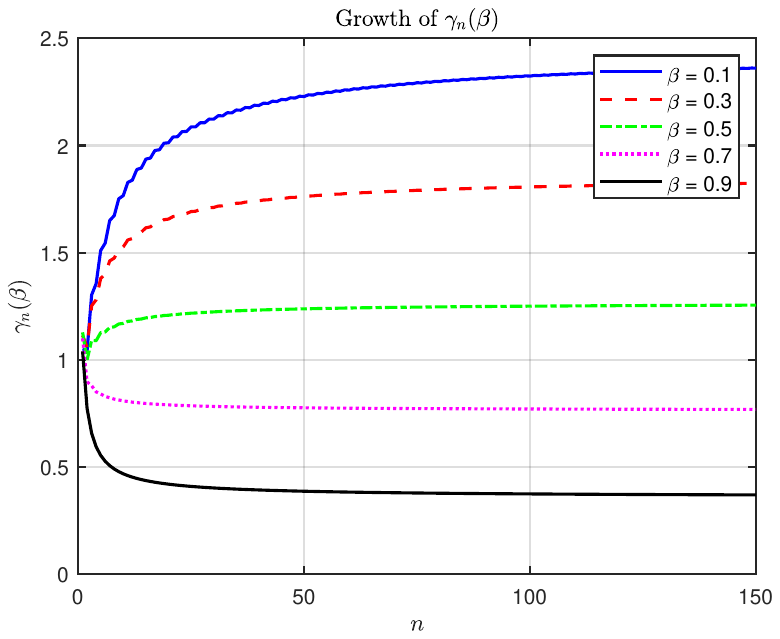} 
         \vspace{-30ex}
    \end{minipage}
    \caption{Left: Radii of convergence ($R$) of $E_c$ and $E_s$. Right: Growth of $\gamma_n(\beta)=\left(\frac{c_n(\beta)}{\Gamma(1+\beta n)}\right)^{1/n}$.}
    \label{growth}
\end{figure}
\begin{figure}[htbp]
    \centering
    \begin{subfigure}[b]{0.4\textwidth}
        \raisebox{8mm}{
        \includegraphics[width=\textwidth]{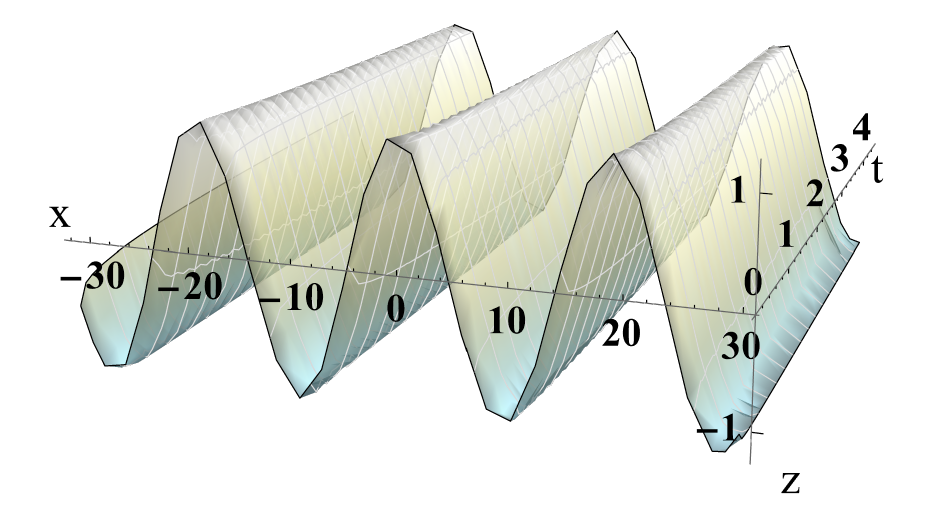}}
    \end{subfigure}
    \hfill 
    \begin{subfigure}[b]{0.4\textwidth}
        \includegraphics[width=\textwidth]{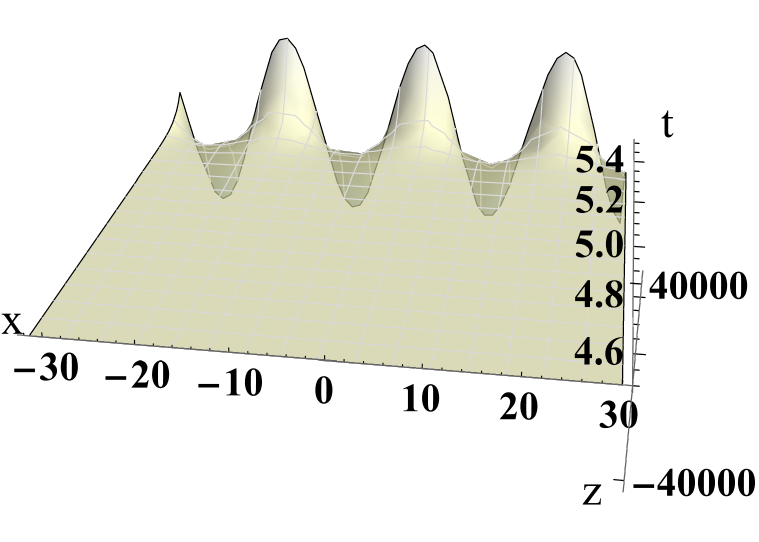}
    \end{subfigure}
    \caption{3D plots for $\beta = 0.3$ with $E_c,E_s$ approximated up to the partial sum $k = 400$. Left: $t \in [0,4.5]$. Right: $t \in [4.5,5.5]$.}
    \label{fig:blowup}
\end{figure}

If $t \ll 1$ is not satisfied, then HPM predicts finite time blow-up. In \Cref{growth}, the computation of higher iterates of $c_n(\beta)$ using MATLAB highly suggests $c_n^{1/n}$ grows as $n^{\beta}$ up to some constant as $n \rightarrow \infty$ by observing the convergent behavior of $\gamma_n = \left(\frac{c_n(\beta)}{\Gamma(1+\beta n)}\right)^{1/n}$ and using the Stirling's approximation to the gamma function. This implies both series $E_c(t;\beta),E_s(t;\beta)$ have the same finite radius of convergence, which we compute numerically in \Cref{growth}. More precisely, from the numerical observation that $\left(\frac{c_n(\beta)}{\Gamma(1+\beta n)}\right)^{\frac{1}{n}}$ converges to a finite value $\gamma(\beta)>0$, the Cauchy-Hadamard's root test is used to compute the radii of convergence. \Cref{fig:blowup} illustrates the divergence of series with $\lim\limits_{t \rightarrow T-} |u(x,t)| = \infty$ for $\beta = 0.3$. The left plot shows sinusoidal waves evolving for small time with $\sim 1$ magnitude, consistent with $u_0 = \left(\frac{3}{2}\right)^{\frac{1}{2}} \cos \left(\frac{x}{3}\right)$. Near the terminal time $T \approx 5.1655$ computed in \Cref{growth}, the right plot shows the amplitude increasing to $\sim 10^4$.
\section{Conclusion}
\begin{itemize}
    \item Showed that the spatial nonlocality modeled by the Riesz derivative prevents the existence of traveling compactons in the Rosenau-Hyman equation by the Paley-Wiener theorem.
    \item Applied HPM to extend the periodic traveling wave solutions of the classical $K(2,2)$ and $K(3,3)$ into the nonlocal regime governed by the Caputo derivative and observed no evidence of traveling waves. 
    \item Showed that the short-time behaviors of the fractional $K(2,2)$ and $K(3,3)$ are similar, marked by the sublinear evolution in phase.
    \item Showed that the long-time dynamics of the fractional $K(2,2)$ exhibits bifurcation as $\beta$ crosses $\frac{1}{2}$.
    \item Numerically computed the coefficients of the infinite series that determines the time evolution of the fractional $K(3,3)$, thereby highly suggesting that the solution cannot be smooth globally in time. 
\end{itemize}
\noindent \textbf{Limitations and future directions:} It is of interest to give a rigorous proof of the super-exponential growth of $c_n(\beta)$, which would imply finite time blow-up in the category of smooth solutions. However there could be non-unique ways to extend the solution by allowing discontinuities, analogous to shock waves in the Burger's equation. Lastly, it was verified that HPM is a powerful and simple method in constructing series solutions. The method was applied to fractional nonlinear PDEs with nonlinear dispersion relations. However, we numerically demonstrated that the method could yield divergent series when applied to the time-fractional KdV. It is left for future research to improve HPM, possibly by combining with VIM, HAM, or ADM, among others, to expand its applicability to wider classes of differential equations.

\bibliographystyle{ieeetr}
\small
\bibliography{ref}

\end{document}